\documentclass{article}

\usepackage[utf8]{inputenc}
\usepackage[utf8]{inputenc}
\usepackage[english]{babel}
\title{A Method to Construct $1$-Rotational Factorizations of Complete Graphs and Solutions to the Oberwolfach Problem}
\author{Daniel McGinnis and Eirini Poimenidou}

\usepackage{tikz}
\usepackage{amssymb}
\usepackage{amsthm}
\usepackage{amsmath}
\newtheorem{theorem}{Theorem}[section]
\newtheorem{definition}[theorem]{Definition}
\newtheorem{lemma}[theorem]{Lemma}

\DeclareMathOperator{\modd}{mod}
\begin{document}
\date{New College of Florida}
\maketitle
\begin{center}
    Abstract
\end{center}
The concept of a $1$-rotational factorization of a complete graph under a finite group $G$ was studied in detail by Buratti and Rinaldi. They found that if $G$ admits a $1$-rotational $2$-factorization, then the involutions of $G$ are pairwise conjugate. We extend their result by showing that if a finite group $G$ admits a $1$-rotational $k=2^nm$-factorization where $n\geq 1$, and $m$ is odd, then $G$ has at most $m(2^n-1)$ conjugacy classes containing involutions. Also, we show that if $G$ has exactly $m(2^n-1)$ conjugacy classes containing involutions, then the product of a central involution with an involution in one conjugacy class yields an involution in a different conjugacy class. We then demonstrate a method of constructing a $1$-rotational $2n$-factorization under $G \times \mathbb{Z}_n$ given a $1$-rotational $2$-factorization under a finite group $G$. This construction, given a $1$-rotational solution to the Oberwolfach problem $OP(a_{\infty},a_1, a_2 \cdots, a_n)$, allows us to find a solution to $OP(2a_{\infty}-1,^2a_1, ^2a_2\cdots, ^2a_n)$ when the $a_i$'s are even ($i \neq \infty$), and $OP(p(a_{\infty}-1)+1, ^pa_1, ^pa_2 \cdots, ^pa_n)$ when $p$ is an odd prime, with no restrictions on the $a_i$'s.

\section{Introduction}

Questions concerning factorizations and decompositions of complete graphs have heavily influenced research in graph theory; one famous problem related to factorizations is the Oberwolfach problem. The Oberwolfach problem, first posed by Gerhard Ringel in $1967$ at a conference in Oberwolfach, Germany, in its original form is as follows. 
\begin{center}
    Given $2n+1$ people and $s$ round tables $T_1,T_2,...T_s$ where table $T_i$ sits $t_i$ people and $t_1+t_2+...+t_s=2n+1$, can you find seating arrangements for $n$ nights so that each person sits next to two different people each night?
\end{center}
The Oberwolfach problem with these parameters is denoted by\\ $OP(t_1,\ t_2, \dots,\ t_s)$. 
Despite the simple statement of the problem, a complete solution is still unknown. However, some cases of the Oberwolfach problem are known, for example, the famous Walecki construction (\cite{alspachwonderful2008}) solves the case when $s=1$ and $T_1=2n+1$, in other words, when there is one round table that sits everybody; a complete solution is also known when the number of tables $s$ is fixed at $2$, \cite{traettacomplete2013}. Some more cases of the Oberwolfach problem are solved in \cite{hiltonsome2001}, \cite{ollissome2005}, \cite{bryantcomplete2009}, \cite{alspachobservations1985}, \cite{alspachoberwolfach1989}, \cite{rinaldigraphproducts2011}, and \cite{alspachon2016}. So far there is only a very small finite number of cases where it is known that the Oberwolfach problem has no solution. In Section $4$ of this paper, we present a method of constructing solutions to the Oberwolfach problem for infinitely many cases when a solution to a single case of the Oberwolfach problem that is said to be $1$-rotational is given. 

As usual, when $V$ is a finite set, we denote the complete graph with vertex-set $V$ by $K_V$. A \textit{$k$-factor} of $K_V$ is a spanning $k$-regular subgraph of $K_V$, and a \textit{$k$-factorization} is a set of $k$-factors whose edges partition $E(K_V)$, the edge set of $K_V$. 

For a finite group $G$, define $\overline{G} = G \cup \{\infty\}$. If $F$ is a $k$-factor of $K_{\overline{G}}$, then for an element $g \in G$, $Fg$ will denote the $k$-factor obtained by multiplying each vertex on the right by $g$, where $\infty g = \infty$, in other words, permuting the vertices of $F$ by the function $a\mapsto ag$. A $k$-factorization $\mathcal{F}$ of $K_{\overline{G}}$ is said to be \textit{$1$-rotational} if for every $k$-factor $F \in \mathcal{F}$ and $g \in G$, $Fg \in \mathcal{F}$. From this point onward, when we speak about a $1$-rotational $k$-factorization of $K_{\overline{G}}$, where $G$ is a finite group, it will be understood that $G$ is acting on the right by right multiplication on the vertex set $\overline{G}$. 

A $2$-factorization of a complete graph into isomorphic $2$-factors is a solution to an associated Oberwolfach problem (this is well known can be seen by associating each cycle of the $2$-factor with a table and the vertices with people). If each $2$-factor is isomorphic to the union of cycles $C_{a_1} \cup C_{a_2} \cup \cdots \cup C_{a_n}$, then we say that the $2$-factorization is a solution to $OP(a_1,a_2, \cdots, a_n)$, and if the $2$-factorization is $1$-rotational under some finite group $G$, then we say that it is a $1$-rotational solution to $OP(a_1,a_2, \cdots, a_n)$. Finally if each $2$-factor of the $2$-factorization contains $k$ copies of say $C_{a_1}$, then we will say that $2$-factorization is a solution to $OP(^ka_1,a_2, \cdots, a_n)$. Of course, this will apply to any of the $a_i$'s.

In \cite{buratti1-rotational2007}, the authors show that if there exists a $1$-rotational $2$-factorization of $K_{\overline{G}}$, then the involutions of $G$ lie in a single conjugacy class. We extend this result by showing that if there exists a $2^nm$-factorization of $K_{\overline{G}}$ where $n\geq 1$ and $m$ is odd, then $G$ has at most $m(2^n-1)$ conjugacy classes containing involutions. We also find that if there exists a $2^n$-factorization of $K_{\overline{G}}$ and $G$ contains exactly $m(2^n-1)$ conjugacy classes of involutions, then multiplication of an involution in one conjugacy class with a central involution will yield an involution lying in different conjugacy class. This then allows us to prove that there exists a $1$-rotational $4$-factorization of $K_{\overline{D_{2N}}}$ if and only if $N \equiv 2$ $\modd(4)$. Finally, given a $1$-rotational $2$-factorization of $K_{\overline{G}}$, in other words given a $1$-rotational solution to an associated Oberwolfach problem, say $OP(a_{\infty},a_1, a_2 \cdots, a_n)$ where $a_{\infty}$ denotes the length of the cycle through $\infty$, we construct a $1$-rotational $2n$-factorization of $K_{\overline{G \times \mathbb{Z}_n}}$. In the case that $n=2$, this construction provides a solution to $OP(2a_{\infty}-1,^2a_1, ^2a_2\cdots, ^2a_n)$ whenever the $a_i$'s ($i \neq \infty$) are even. In the case that $n=p$ is an odd prime, the construction leads to a solution of $OP(p(a_{\infty}-1)+1, ^pa_1, ^pa_2 \cdots, ^pa_n)$ regardless of the $a_i$'s.

\section{$k$-Starters and $1$-Rotational $k$-Factorizations}
Throughout this paper two vertices in square brackets will denote an edge, for example $[a,b]$ will be the edge with endpoints $a$ and $b$, and two elements in parenthesis (for example $(a,b)$) will denote an ordered pair. If $[a,b]$ is an edge of $K_{\overline{G}}$, then $[a,b]g$ will denote the edge $[ag,bg]$, where $ag$ for instance denotes multiplication of $a$ and $g$ under the operation of $G$. Also if $u$ and $v$ are vertices of some graph we will use the notation $u \sim v$ to mean that ``$u$ is adjacent to $v$''.

We will use the notation $C_n$ to denote a cycle of length $n$. If $x_1, x_2, \cdots ,x_n$ are understood to be vertices of a graph, then $(x_1,x_2, \cdots , x_n)$ and $[x_1,x_2, \cdots , x_n]$ will denote, respectively, a cycle and a path with vertices the $x_i$'s. 

Finally, notice that if there exists a $1$-rotational $k$-factorization of $K_{\overline{G}}$, then the order of the group $G$ is equal to the degree of any vertex of $K_{\overline{G}}$ ($|G|=|\overline{G}|-1$), so it follows that if $\mathcal{F}$ is a $1$-rotational $k$-factorization of $K_{\overline{G}}$, then $|G|=k|\mathcal{F}|$. Therefore, $k$ must divide the order of $G$. 

Now, we give some definitions and an important theorem which were all presented in \cite{buratti1-rotational2007}.
\begin{definition}
 Let $G$ be a finite group and let $\Gamma$ be a simple graph whose vertices are the elements of $\overline{G}$. The {\em list of differences}, $\Delta \Gamma$, is the multiset 
\[
\Delta \Gamma = \{ab^{-1}, ba^{-1} \ : \ [a,b] \in E(\Gamma) \quad a \neq \infty \neq b\}.
\]
If $x \in \Delta \Gamma$, then we say $\Delta \Gamma$ covers $x$.
\end{definition}
\begin{definition}\label{strtrdef}
{\em [\cite{buratti1-rotational2007}, Definition 2.2]} Let $G$ be a finite group, whose order is divisible by $k$ (or a group of odd order if $k=1$), and let $F$ be a $k$-factor of $K_{\overline{G}}$. We say $F$ is a {\em $k$-starter} under $G$ if the following are satisfied: \\ \\ 
1. The $G$-stabilizer of $F$ has order $k$. \\ 
2. $\Delta \Gamma$ covers every element of $G-\{1_G\}$.
\end{definition}
\begin{theorem}\label{thm1}
{\em [\cite{buratti1-rotational2007}, Theorem 2.3]} 
A $1$-rotational $k$-factorization of $K_{\overline{G}}$ is equivalent to the existence of a $k$-starter in $G$.
\end{theorem}
More precisely the authors of \cite{buratti1-rotational2007} show that if $\mathcal{F}$ is a $1$-rotational \\ $k$-factorization of $K_{\overline{G}}$, then any $k$-factor $F \in \mathcal{F}$ is a $k$-starter in $G$ and the $G$-orbit of a $k$-starter under $G$ is a $1$-rotational $k$-factorization of $K_{\overline{G}}$, which in particular means that the $k$-factors in a $1$-rotational $k$-factorization will be pairwise isomorphic.

Throughout the rest of the paper will say that a finite group $G$ is $R_k$ if there exists a $1$-rotational $k$-factorization of $K_{\overline{G}}$.

If a group $G$ is $R_k$ and $k$ is even, then the order of $G$ is even as well. In what follows we suppose $G$ to be a group of even order which is $R_k$. We will prove that $k$ is necessarily even. We will also present some necessary conditions on the conjugacy classes containing involutions.

\begin{lemma}\label{lem1}
Let $\mathcal{F}$ be a $1$-rotational $k$-factorization of $K_{\overline{G}}$, where the order of $G$ is even. Then $k$ is even and every involution of $G$ is contained in the $G$-stabilizer of some factor $F \in \mathcal{F}$.
\end{lemma}
\begin{proof}
Let $x$ be an involution of $G$ and let $F \in \mathcal{F}$ be the factor that contains the edge $[1_G, x]$. Then $Fx$ also contains the edge $[1_G, x]x=[1_G,x]$. Since $\mathcal{F}$ is a decomposition of $K_{\overline{G}}$ and $Fx \in \mathcal{F}$, it must be true that $Fx = F$. Therefore $x$ is contained in the $G$-stabilizer of $F$. The $G$-stabilizer has order $k$, hence, $k$ is even.
\end{proof}

Generalizing the proof of Theorem 3.5 of \cite{buratti1-rotational2007}, we obtain the following statement.

\begin{theorem}\label{generalizedthm}
If $G$ is $R_{2^nm}$ where $n\geq 1$, and $m$ is odd, then $G$ has at most $m(2^n-1)$ conjugacy classes containing involutions.
\end{theorem}
\begin{proof}
Let $\mathcal{F}$ be a $1$-rotational $k$-factorization of $K_{\overline{G}}$, and let $F$ be a factor in $\mathcal{F}$, denoting it's $G$-stabilizer as $S$. We show that $S$ contains at least one representative from every conjugacy class of $G$ containing involutions. Let $C$ be a conjugacy class containing involutions, and let $x$ be an element of $C$. Since the action of $G$ on the factors of $\mathcal{F}$ is transitive, the $G$-stabilizers of the factors of $\mathcal{F}$ are pairwise conjugate. By Lemma \ref{lem1}, there exists a factor $F'$ whose stabilizer $S'$ contains $x$. Since $S=g^{-1}S'g$ for some $g \in G$, $S$ contains the element $g^{-1}xg$, which is a representative of $C$. Because $|S|=k=2^nm$ there are at most $m$ Sylow $2$-subgroups of $S$. Since every involution is in some Sylow $2$-subgroup, $S$ contains at most $m(2^n-1)$ involutions, hence $G$ has at most $m(2^n-1)$ conjugacy classes containing involutions.
\end{proof}

\begin{theorem}\label{conjthm}
If $G$ is $R_{2^nm}$ where $n\geq 1$ and $m$ is odd, and $G$ has $m(2^n-1)$ conjugacy classes of involutions, then the multiplication of any central involution with an involution in one conjugacy class yields an involution in a different conjugacy class.
\end{theorem}

\begin{proof}
Let $\mathcal{F}$ be a $1$-rotational $2^nm$-factorization of $K_{\overline{G}}$. Given an involution $x$ of $G$, we have by Lemma \ref{lem1} that there exists a factor $F \in \mathcal{F}$ whose $G$-stabilizer $S$ contains $x$. 

Notice that if $z$ is a central involution, then $z \in S$ since $z$ is conjugate only to itself. Because $S$ contains at least one representative from each conjugacy class of involutions and because $G$ has $m(2^n-1)$ conjugacy classes of involutions, it follows from the fact that $S$ has at most $m(2^n-1)$ involutions that $S$ contains precisely one representative from each conjugacy class of involutions. Thus, $xz$ lies in a different conjugacy class from $x$ since $xz \in S$.
\end{proof}

\section{A $1$-Rotational $4$-Factorization of $K_{\overline{D_{4N}}}$}
Let $D_{2N}$ be the dihedral group of order $2N$, namely the group with the defining relations: 
\[
\langle r,s : r^{N}=s^2=1 , srs=r^{-1} \rangle.
\]
\begin{theorem}
$D_{2N}$ is $R_4$ if and only if $N \equiv 2$ $\modd(4)$.
\end{theorem}
\begin{proof}

Let $D_{2N}$ be $R_4$, then clearly $4||D_{2N}|=2N$, so either $N \equiv 0$ $\modd (4)$ or $N \equiv 2$ $\modd(4)$. However, if $N \equiv 0$ $\modd(4)$, then the central element $r^{N/2}$ is an even power of $r$ so $r^{N/2}s$ lies in the same conjugacy class as $s$. This is a contradiction to Theorem \ref{conjthm}, proving one direction of our claim.

Now let $N \equiv 2$ $\modd(4)$, and consider the $2$-starter $F$ under $D_{N}$ presented in Theorem 5.1 of \cite{buratti1-rotational2007} whose stabilizer is $\{1,s\}$. Let $H$ be the subgraph of $K_{\overline{D_{2N}}}$ with the edge-set:
\begin{align*}
E(H)&=\{[x,y],[x,r^{N/2}y]\ :\ [x,y]\in E(F),\ x,y\neq \infty\}\\
&\cup \{[\infty,1],[\infty,s],[\infty,r^{N/2}],[\infty,r^{N/2}s],[s,r^{N/2}s],[1,r^{N/2}]\},
\end{align*}
where an element $r^is^j$ ($i\in \{0,\dots,N/2-1\}$, $j\in \{0,1\}$) of $D_N$ will simply be the element $r^is^j$ when regarded as an element of $D_{2N}$.
It is easy to verify that $H$ is a $4$-starter under $D_{2N}$ whose stabilizer is $\{1,s,r^{N/2},r^{N/2}s\}$, thus $D_{2N}$ is $R_4$. 
\end{proof}
\theoremstyle{definition}
\newtheorem{example}{Example}[section]

\section{Using a $2$-Starter of $K_{\overline{G}}$ to Obtain a $2n$-Starter of $K_{\overline{G \times \mathbb{Z}_n}}$ and New Solutions to the Oberwolfach Problem}
In the following theorem we provide a construction for a $2n$-starter of $K_{\overline{G \times \mathbb{Z}_n}}$ given a $2$-Starter of $K_{\overline{G}}$. Here we will denote the operation in $G$ as regular group multiplication and the operation in $\mathbb{\mathbb{Z}}_n$ as addition. Also we will refer to the elements of $G \times \mathbb{Z}_n$ as ordered pairs $(g,k)$ where $g \in G$ and $k \in \mathbb{Z}_n$.
\begin{theorem}\label{2nthm}
Given a $2$-starter of $K_{\overline{G}}$, there exists a $2n$-starter of $K_{\overline{G \times \mathbb{Z}_n}}$.
\end{theorem}
\begin{proof}
Let $F$ be a $2$-starter of $K_{\overline{G}}$ with $G$-stabilizer say $S$, and let $a$ and $b$ be the vertices adjacent to $\infty$. Now consider the following subgraph $H$ of $K_{\overline{G \times \mathbb{Z}_n}}$
\begin{align*}
 &\{[\infty,\ (x,\ k)]:x \sim \infty \textrm{ in } F \textrm{ and } k \in \mathbb{Z}_n\} \\
\cup\ &\{[(x,\ k),\ (x,\ r)]:x \sim \infty \textrm{ in } F \textrm{ and } k,\ r \in \mathbb{Z}_n,\ k \neq r\}\\
\cup\ &\{[(x,\ k),\ (y,\ r)]:x \sim y \textrm{ in } F \textrm{ and } k,\ r \in \mathbb{Z}_n\}.
\end{align*}

First we show that $H$ is a $2n$-factor. If $x$ is an element not adjacent to $\infty$, then $x$ is adjacent to precisely two distinct elements, say $y$ and $z$, of $G$ in $F$. Therefore for each $k \in \mathbb{Z}_n$, $(x,\ k)$ is adjacent to vertices of the form $(y,\ r)$ for all $r \in \mathbb{Z}_n$, and vertices of the form $(z,\ t)$ for all $t \in \mathbb{Z}_n$. This shows that $(x,\ k)$ has degree $2n$ for all $k$. 

Also $(a,\ k)$ is adjacent to vertices of the form $(b,\ r)$ for all $r \in \mathbb{Z}_n$, $(a,\ t)$ where $t \in \mathbb{Z}_n$ and $t \neq k$, and $\infty$ showing that $(a,\ k)$ has degree $2n$ for all $k$. The same argument shows that $(b,\ k)$ has degree $2n$ for all $k$, and clearly $\infty$ has degree $2n$. Therefore $H$ is a $2n$-factor.

To check property one of a $2n$-starter (see Definition \ref{strtrdef}) denote the $G \times \mathbb{Z}_n$-stabilizer of $H$ as $S'$; we will show that $S'=S \times \mathbb{Z}_n$. Notice that every element of $G \times \mathbb{Z}_n$ of the form $(1_G,\ k)$ fixes $H$. Also, if $s$ is the nontrivial element in $S$ (recall that $S$ is the $G$-stabilizer of $F$), then the element $(s,\ 0)$ fixes $H$.

Therefore, $S \times \mathbb{Z}_n \subset S'$. To show $S \times \mathbb{Z}_n = S'$, let $x$ be an element of $G-S$. If $(x,\ k) \in S'$ for some $k$ then $(x,\ k)(1_G,\ k)^{-1}(x,\ 0) \in S'$. However because $x \notin S$, there is some edge $[u,\ v] \in F$ such that $[u,\ v]x$ does not lie in $F$, but this implies that the edge $[(u,\ 0),\ (v,\ 0)](x,\ 0)$ does not lie in $H$. This is a contradiction, thus, we have proven that $S \times \mathbb{Z}_n = S'$. Since $|S \times \mathbb{Z}_n| = 2n$, it has been verified that $H$ satisfies property one of a $2n$-starter.

To check property two of a $2n$-starter let $(x,\ k) \in (G \times \mathbb{Z}_n)-(1_G,\ 0)$ where $x \neq 1_G$. There exists an edge $[w,\ p]$ of $F$ such that $wp^{-1}=x$, so the edge $[(w,\ k),\ (p,\ 0)]$ of $H$ contributes the value $(w,\ k)(p,\ 0)^{-1}=(wp^{-1},\ k)=(x,\ k)$ to $\Delta H$. If $x = 1_G$, then the edge $[(a,\ k),\ (a,\ 0)]$ contributes $(1_G,\ k)=(x,\ k)$ to $\Delta H$. Therefore $\Delta H$ covers $(G \times \mathbb{Z}_n)-(1_G,\ 0)$, showing that $H$ satisfies property two of a $2n$-starter. This along with the above proves that $H$ is indeed a $2n$-starter.
\end{proof}

We will now consider a $2$-starter $F$ under a finite group $G$; $F$ is isomorphic to a union of cycles $C_{a_{\infty}} \cup C_{a_1} \cup C_{a_2} \cup \cdots \cup C_{a_N}$. Let $m=a_{\infty}-3$ (notice $m$ is even), we will write $F$ more explicitly as the union of cycles
\[
(\infty,\ a,\ y_1,\ y_2, \cdots,\ y_m,\ b)
\]
\[
\cup \ (x^1_1,\ x^1_2, \cdots,\ x^1_{a_{1}})
\]
$$\vdots$$
\[
\cup \ (x^N_1,\ x^N_2, \cdots,\ x^N_{a_{N}}).
\]
Let $p$ be a prime, and consider the $2p$-starter $H$ under $G \times \mathbb{Z}_p$ obtained from applying Theorem \ref{2nthm} to $F$. We will show that $H$ can be decomposed into $p$ isomorphic $2$-factors when $p$ is an odd prime, and $H$ can be decomposed into $2$ isomorphic $2$-factors when $p=2$ and the $a_i$'s ($i \neq \infty$) are even. 

Let $r_i$ be the smallest non-negative number such that $r_i \equiv a_i$ $\modd(p)$, if $r_i \not\equiv 1$ $\modd(p)$, then for $i \in \{1,2, \ldots, N\}$, $j \in \{0,1, \ldots, p-1\}$ define \newline \newline
$
F_{ij} =
$
\[
\bigcup_{k=0}^{p-1} ((x_1^i,k),(x_2^i,j+k),(x_3^i,2j+k),...,(x_{a_i-1}^i,(r_i-2)j+k),(x_{a_i}^i,(r_i-1)j+k)).
\]
If $r_i \equiv 1 $ $\modd(p)$, then define \newline \newline
$
F_{ij} =
$
\[
\bigcup_{k=0}^{p-1} ((x_1^i,k),(x_2^i,j+k),(x_3^i,2j+k),...,(x_{a_i-1}^i,(p-1)j+k),(x_{a_i}^i,(p-2)j+k)).
\]
Notice that each $F_{ij}$ is isomorphic to $\bigcup_{t=1}^pC_{a_i}$. 
Let $H_{\infty}$ be the complete graph with vertex set 
\[
\{(a,k),\ (b,k) : k \in \mathbb{Z}_p\}\cup\{\infty\}.
\] 
Thus we can use the Walecki construction to decompose $H_{\infty}$ into $p$ cycles of length $2p+1$; we will do this explicitly. Arrange the vertices of $H_{\infty}$ except $\infty$ in a circle counterclockwise in this order 
\[
(a,0),\ (b,0),\ (a,1),\ (b,1),\ (a,2),\ (b,2),\ \dots,\ (a,p-1),\ (b,p-1).
\]
First consider the $2p+1$-cycle
\begin{align*}
&E_0= \\
&((a,0),(b,0),(b,p-1),(a,1),(a,p-1),(b,1), \dots, (a,\tfrac{p-1}{2}+1),(b,\tfrac{p-1}{2}),\infty),
\end{align*}
and define $E_i$ as the $2p+1$-cycle obtained by ticking each vertex of $E_0$ that is not the $\infty$ vertex counterclockwise $i$ times. For example, 
\begin{align*}
&E_1= \\
&((b,0),(a,1),(a,0),(b,1),(b,p-1),(a,2), \dots, (b,\tfrac{p-1}{2}+1),(a,\tfrac{p-1}{2}+1),\infty).
\end{align*}
Since the $E_i$'s were constructed by the Walecki construction, the $E_i$'s form a decomposition of $H_{\infty}$. 

Now define $\overline{H_{\infty}}$ as the induced subgraph of $H$ with the vertex set 
\[
\{(a,\ k),\ (b,\ k),\ (y_i,\ k)\ :\ i \in \{1,\ \dots,\ m\},\ k \in \mathbb{Z}\}
.\]
We will find a decomposition of $\overline{H_{\infty}}$ into cycles of length $p(a_{\infty}-1)+1$. Consider the $2p+1$-cycle $E_i$; whenever two vertices of the form $(a,\ k)$ and $(b,\ t)$ are adjacent in $E_i$, insert the path 
\[
[(y_1,\ t),\ (y_2,\ k),\ (y_3,\ t),\ \dots \ , (y_m,\ k)]
\]
(the second component of the ordered pairs alternate between $t$ and $k$) between $(a,\ k)$ and $(b,\ t)$. Denote the resulting cycle as $\overline{E_i}$. In other words if
\[
E_i=(\dots,\ (a,\ k),\ (b,\ t),\ \dots,\ \infty),
\]
then
\[
\overline{E_i}=(\dots,\ (a,\ k),\ (y_1,\ t),\ (y_2,\ k),\ (y_3,\ t),\ \dots \ ,\ (y_m,\ k),\ (b,\ t),\ \dots,\ \infty).
\]
Or if
\[
E_i=(\dots,\ (b,\ t),\ (a,\ k),\ \dots,\ \infty),
\]
then
\[
\overline{E_i}=(\dots,\ (b,\ t),\ (y_m,\ k),\ (y_{m-1},\ t),\ (y_{m-2},\ k),\ \dots \ ,\ (y_1,\ t),\ (a,\ k),\ \dots,\ \infty).
\]
For example, 
\begin{align*}
&\overline{E_0}=\\ 
&((a,\ 0),\ (y_1,\ 0),\ (y_2,\ 0),\ (y_3,\ 0),\ \dots,\ (y_m,\ 0),\ (b,\ 0),\ \\ 
&(b,\ p-1),\ (y_m,\ 1),\ (y_{m-1},\ p-1),\ (y_{m-2},\ 1), \dots, (y_1,\ p-1),\ (a,\ 1),\ \\ 
&(a,\ p-1),\ (y_1,\ 1),\ (y_2,\ p-1),\ (y_3,\ 1),\ \dots,\ (y_m,\ p-1),\ (b,\ 1),\ \dots, \\ 
&(a, \tfrac{p-1}{2}+1),(y_1, \tfrac{p-1}{2}), (y_2, \tfrac{p-1}{2}+1), (y_3, \tfrac{p-1}{2}), \dots, (y_m, \tfrac{p-1}{2}+1), (b, \tfrac{p-1}{2}), \infty).
\end{align*}
Notice that each $\overline{E_i}$ is a cycle of length $p(a_{\infty}-1)+1$. Also the $\overline{E_i}$'s are edge disjoint (we show this in the corollary below), so they form a decomposition of $\overline{H_{\infty}}$.

Now let 
\[
H_j= \overline{E_j}\ \cup  \biggl{\{}\bigcup_{i=1}^N\ F_{ij}\biggr{\}}
\]
for all $j \in \{0,\ \dots,\ p-1\}$. In the case that $p$ is an odd prime or $p=2$ and the $a_i$'s ($i\neq \infty$) are even,
each $H_j$ is a $2$-factor isomorphic to 
\[
C_{p(a_{\infty}-1)+1}\ \cup\ \biggl{\{}\bigcup_{k=1}^pC_{a_1}\biggr{\}}\ \cup\ \cdots \cup\ \biggl{\{}\bigcup_{k=1}^pC_{a_N}\biggr{\}},
\]
and the set $\{H_j\}_{j=0}^{p-1}$ forms a decomposition of $H$. This leads us to the following result where we also verify this fact.
\begin{theorem}\label{cor2}
Assume there exists a $1$-rotational solution to $\\ OP(a_{\infty},\ a_1,\  a_2,\  \cdots, a_N)$ under some finite group $G$. If the $a_i$'s ($i\neq \infty$) are even, then there exists a solution to $OP(2a_{\infty}-1, ^2a_1, ^2a_2, \cdots, ^2a_N)$. If $p$ is an odd prime, then there exists a solution to $OP(p(a_{\infty}-1)+1,\ ^pa_1,\ ^pa_2,\ \cdots,\ ^pa_N)$.
\end{theorem}
\begin{proof}
The same notation introduced in the above discussion will be used \\ throughout this proof. 

First consider the case where the $a_i$'s ($i\neq \infty$) are even. Applying the above discussion for $p=2$, it can easily be verified that $\{H_0, H_1\}$ is a decomposition of $H$ $2$-factors isomorphic to $C_{2a_{\infty}-1} \cup \{C_{a_1} \cup C_{a_1}\} \cup \{C_{a_2} \cup C_{a_2}\} \cup \cdots \cup \{C_{a_N} \cup C_{a_N}\}$. Now we can take the $4$-factorization $\mathcal{F}$ of $K_{\overline{G \times \mathbb{Z}_2}}$ that is the $G \times \mathbb{Z}_2$-orbit of $H$ and decompose each $4$-factor of $\mathcal{F}$ into two $2$-factors isomorphic to $C_{2a_{\infty}-1} \cup \{C_{a_1} \cup C_{a_1}\} \cup \{C_{a_2} \cup C_{a_2}\} \cup \cdots \cup \{C_{a_N} \cup C_{a_N}\}$. This then provides us with a $2$-factorization of $K_{\overline{G \times \mathbb{Z}_2}}$, and  a solution to $OP(2a_{\infty}-1, ^2a_1, ^2a_2, \cdots, ^2a_N)$.

Now consider the case that $p$ is an odd prime. First we will show that $F_{qj}$ is edge disjoint from $F_{si}$ if $q \neq s$ or $i \neq j$. It is clear that if $q \neq s$, then $F_{qj}$ is edge disjoint from $F_{si}$ for all $i$ and $j$ since, in this case, $F_{qj}$ and $F_{si}$ do not share any vertices in common. 

Notice that the fact that $F_{qj}$ and $F_{sj}$ are vertex disjoint when $q \neq s$ justifies the fact that each $H_j$ defined above is a union of cycles as we had stated.

If $q=s$ then $F_{qi}$ is edge disjoint from $F_{si}=F_{qj}$ for all $i \neq j$ as follows. 

Let $(x_t^q,\ k)$ be a vertex in $F_{qi}$ where $2 \leq t \leq a_q-1$ and $0 \leq k \leq p-1$. If $a_q \not\equiv 1$ $\modd(p)$, then $(x_t^q,\ k)$ is adjacent to $(x_{t-1}^q,\ k-i)$ and $(x_{t+1}^q,\ k+i)$ in $F_{qi}$. Similarly, $(x_t^q,\ k)$ is adjacent to $(x_{t-1}^q,\ k-j)$ and $(x_{t+1}^q,\ k+j)$ in $F_{qj}$. Since $k-i \not\equiv k-j$ $\modd(p)$ and $ k+i \not\equiv k+j$ $\modd(p)$, the edges connected to $(x_t^q,\ k)$ are different in $F_{qi}$ from $F_{qj}$. 

If $a_q \equiv 1$ $\modd(p)$, then the above holds except when $t=a_q-1$, in this case $(x_t^q,\ k)$ is adjacent to $(x_{t-1}^q,\ k-i)$ and $(x_{t+1}^q,\ k-i)$ in $F_{qi}$ and adjacent to $(x_{t-1}^q,\ k-j)$ and $(x_{t+1}^q,\ k-j)$ in $F_{qj}$, but the same argument still follows. 

If $t = 1$ and $a_q \not\equiv 1$ $\modd(p)$, then $(x_t^q,\ k)=(x_1^q,\ k)$ is adjacent to $(x_{a_q}^q,\ (r_q-1)j+k)$ and $(x_2^q,\  j+k)$ in $F_{qj}$, and $(x_1^q,\ k)$ is adjacent to $(x_{a_q}^q,\ (r_q-1)i+k)$ and $(x_2^q,\  i+k)$ in $F_{qi}$. Clearly $i+k \not\equiv j+k$ $\modd(p)$ and if $(r_q-1)i+k \equiv (r_q-1)j+k$ $\modd(p)$, then $(r_q-1)i \equiv (r_q-1)j$ $\modd(p)$ which implies that $i \equiv j$ $\modd(p)$ (since $gcd(r_q-1,\ p)=1$). This a contradiction, so the edges connected to $(x_1^q,\ k)$ are different in $F_{qi}$ than in $F_{qj}$. A similar argument can be made when $t=a_q$. 

When $a_q \equiv 1$ $\modd(p)$, $(x_1^q,\ k)$ is adjacent to $(x_{a_q}^q,\ (p-2)j+k)$ and $(x_2^q,\  j+k)$ in $F_{qj}$, and adjacent to $(x_{a_q}^q,\ (p-2)i+k)$ and $(x_2^q,\  i+k)$ in $F_{qi}$. A similar argument as sbove then shows that the edges connected to $(x_1^q,\ k)$ are different in $F_{qi}$ than in $F_{qj}$. A similar argument also follows for the vertex $(x_{a_q}^q,\ k)$. It follows that $F_{qi}$ and $F_{sj}$ are edge disjoint whenever $q \neq s$ or $i \neq j$.

Now we show that $\overline{E_j}$ is edge disjoint from $\overline{E_i}$ when $i \neq j$. Let $i \neq j$ and let $2 \leq l \leq m-1$,\  $k \in \mathbb{Z}_p$ and consider the vertex $(y_l,\ k)$. Let $(y_l,\ k)$ be adjacent to $(y_{l-1},\ t)$ and $(y_{l+1},\ t)$ in $\overline{E_j}$, then $(a,\ k)$ and $(b,\ t)$ are adjacent in $E_i$. Since $E_j$ is edge disjoint from $E_i$, $(a,\ k)$ and $(b,\ t)$ are not adjacent in $E_j$. Therefore, $(y_l,\ k)$ is not adjacent to $(y_{l-1},\ t)$ or $(y_{l+1},\ t)$ in $\overline{E_i}$. A similar argument can be made for vertices of the form $(a,\ k)$, $(b,\ k)$, $(y_1,\ k)$, and $(y_m,\ k)$, showing that $\overline{E_j}$ is edge disjoint from $\overline{E_i}$ when $i \neq j$. 

It then follows that that $H_i$ is edge disjoint from $H_j$ whenever $i \neq j$. Also, it can easily be checked that each edge of $H$ is contained in some $H_i$. Therefore, $\{H_k\}_{k=0}^{p-1}$ is a decomposition of $H$, so $H$ is decomposable into $p$, $2$-factors isomorphic to
\[
C_{p(a_{\infty}-1)+1}\ \cup\ \biggl{\{}\bigcup_{k=1}^pC_{a_1}\biggr{\}}\ \cup\ \cdots \ \cup\ \biggl{\{}\bigcup_{k=1}^pC_{a_N}\biggr{\}}.
\]
Thus, we can take all $2p$-factors of $H$ under the action of $G \times \mathbb{Z}_p$ and decompose each into $p$ copies of $C_{p(a_{\infty}-1)+1}\ \cup\ \{\bigcup_{k=1}^pC_{a_1}\}\ \cup\ \cdots \cup\ \{\bigcup_{k=1}^pC_{a_N}\}$. This will result in a $2$-factorization of $K_{\overline{G \times \mathbb{Z}_p}}$ into $2$-factors that are all isomorphic to $C_{p(a_{\infty}-1)+1}\ \cup\ \{\bigcup_{k=1}^pC_{a_1}\}\ \cup\ \cdots \cup\ \{\bigcup_{k=1}^pC_{a_N}\}$, and hence, will provide us with a solution to $OP(p(a_{\infty}-1)+1,\ ^pa_1,\ ^pa_2,\ \cdots,\ ^pa_N)$.
\end{proof}

\begin{example}
Consider the following $2$-starter 
\[
F = (\infty,\ 0,\ 3,\ 9,\ 6) \cup (1,\ 5,\ 4,\ 2,\ 7,\ 11,\ 10,\ 8)
\] 
under $\mathbb{Z}_{12}$. We can apply Theorem \ref{2nthm} for $p=3$ to this $2$-starter and obtain a $6$-starter under $\mathbb{Z}_{12} \times \mathbb{Z}_3$, say $H$. We demonstrate how $H$ can be decomposed into three isomorphic $2$-factors. Throughout this example, we will use the same notation as the discussion preceding Theorem \ref{cor2}. 

The $F{ij}$'s are 
\begin{align*}
    F_{10} = &((1,\ 0),\ (5,\ 0),\ (4,\ 0),\ (2,\ 0),\ (7,\ 0),\ (11,\ 0),\ (10,\ 0),\ (8,\ 0)) \\
    \cup\ &((1,\ 1),\ (5,\ 1),\ (4,\ 1),\ (2,\ 1),\ (7,\ 1),\ (11,\ 1),\ (10,\ 1),\ (8,\ 1)) \\
    \cup\ &((1,\ 2),\ (5,\ 2),\ (4,\ 2),\ (2,\ 2),\ (7,\ 2),\ (11,\ 2),\ (10,\ 2),\ (8,\ 2)),
\end{align*}
\begin{align*}
    F_{11} = &((1,\ 0),\ (5,\ 1),\ (4,\ 2),\ (2,\ 0),\ (7,\ 1),\ (11,\ 2),\ (10,\ 0),\ (8,\ 1)) \\
    \cup\ &((1,\ 1),\ (5,\ 2),\ (4,\ 0),\ (2,\ 1),\ (7,\ 2),\ (11,\ 0),\ (10,\ 1),\ (8,\ 2)) \\
    \cup\ &((1,\ 2),\ (5,\ 0),\ (4,\ 1),\ (2,\ 2),\ (7,\ 0),\ (11,\ 1),\ (10,\ 2),\ (8,\ 0)),
\end{align*}
\begin{align*}
    F_{12} = &((1,\ 0),\ (5,\ 2),\ (4,\ 1),\ (2,\ 0),\ (7,\ 2),\ (11,\ 1),\ (10,\ 0),\ (8,\ 2)) \\
    \cup\ &((1,\ 1),\ (5,\ 0),\ (4,\ 2),\ (2,\ 1),\ (7,\ 0),\ (11,\ 2),\ (10,\ 1),\ (8,\ 0)) \\
    \cup\ &((1,\ 2),\ (5,\ 1),\ (4,\ 0),\ (2,\ 2),\ (7,\ 1),\ (11,\ 0),\ (10,\ 2),\ (8,\ 1)).
\end{align*}
The $E_i$'s that form a decomposition of $H_{\infty}$ are
\begin{align*}
    &E_0 = ((0,\ 0),\ (6,\ 0),\ (6,\ 2),\ (0,\ 1),\ (0,\ 2),\ (6,\ 1),\ \infty) \\ 
    &E_1 = ((6,\ 0),\ (0,\ 1),\ (0,\ 0),\ (6,\ 1),\ (6,\ 2),\ (0,\ 2),\ \infty) \\ 
    &E_2 = ((0,\ 1),\ (6,\ 1),\ (6,\ 0),\ (0,\ 2),\ (0,\ 0),\ (6,\ 2),\ \infty).
\end{align*}
From this, we get that the $\overline{E_i}$'s that form a decomposition of $\overline{H_{\infty}}$ are
\begin{align*}
    &\overline{E_0} = \\ &((0,0),(3,0),(9,0),(6,0),(6,2),(9,1),(3,2),(0,1),(0,2),(3,1),(9,2),(6,1),\infty) \\
    &\overline{E_1} = \\ &((6,0),(9,1),(3,0),(0,1),(0,0),(3,1),(9,0),(6,1),(6,2),(9,2),(3,2),(0,2),\infty) \\
    &\overline{E_2} = \\ &((0,1),(3,1),(9,1),(6,1),(6,0),(9,2),(3,0),(0,2),(0,0),(3,2),(9,0),(6,2),\infty).
\end{align*}
Notice that each $F_{1j}$ is isomorphic to $\{ \cup_{k=1}^3 C_8 \}$, and each $\overline{E_i}$ is a $3(5-1) + 1 = 13$-cycle. Thus we have that the set $\{F_{1j} \cup \overline{E_j} \}_{j=0}^2$ is a decomposition of $H$ into three $2$-factors isomorphic to $\{ \cup_{k=1}^3 C_8 \} \cup C_{13}$. 
\end{example}
Now we apply Theorem \ref{cor2} to solve some explicit Oberwolfach problems. \\ \\
\textbf{Application 1:}\\
In \cite{buratti1-rotational2007} the authors found a $1$-rotational solution to $OP(3,\ ^k4)$ under $D_{4k+2}$ for each $k \in \mathbb{N}$. Thus, by Theorem \ref{cor2}, we obtain a solution to $OP(2p+1,\ ^{kp}4)$ for each prime $p$ and each $k \in \mathbb{N}$.

In the same paper, it is shown that there is a $1$-rotational solution to $OP(3,\ 6,\ ^{2k-2}4)$ under the dicyclic group $Q_{8k}$ for each $k \in \mathbb{N}$. Again we can apply Theorem \ref{cor2} to obtain a solution to $OP(2p+1,\ ^p6,\ ^{p(2k-2)}4)$ for each prime $p$ and each $k \in \mathbb{N}$.\\ \\
\textbf{Application 2:}\\
Collecting results from Propositions 3.9 and 3.10 and Theorem 3.13 from \cite{traettacomplete2013}, Theorem 6.9 from \cite{buratti2-starters2012}, and Theorem 4.2 from \cite{burattinon-existence2009}, we have that there exists a $1$-rotational solution to $OP(2r+1,\ 2s)$ except when $(2r+1,\ 2s) = (3,\ 8),\ (5,\ 4),$ $(7,\ 4),\ (5,\ 6)$ or $(3,\ 2s)$ where $s \equiv 0 \textrm{ or } 1\ \modd(4)$. Thus, applying Theorem \ref{cor2}, we obtain a solution to $OP(2pr+1,\ ^p2s)$ for all primes $p$ with the same conditions on the values for $(2r+1,\ 2s)$.\\ \\
\textbf{Application 3:}\\
In \cite{buratti2-starters2012}, the authors found a method to construct $1$-rotational solutions to many different Oberwolfach problems. In particular, it is shown that for all $\alpha \geq 1$ and $k \geq 3$ there exists a $1$-rotational solution to $OP(2^{\alpha}+1,\ ^2k,\ ^{2^{\alpha}-1}(2k))$ when $k$ is odd and $OP(2^{\alpha}+1,\ ^{2^{\alpha}+1}k)$ when $k$ is even (Proposition 3.9 in \cite{buratti2-starters2012}). Therefore we obtain solutions to $OP(p2^{\alpha}+1,\ ^{2p}k,\ ^{p2^{\alpha}-1}(2k))$ for each odd prime $p$ and $OP(p2^{\alpha}+1,\ ^{p(2^{\alpha}+1)}k)$ for all primes $p$ after applying Theorem \ref{cor2} with the same conditions on $\alpha$ and $k$.

In summary, the following Oberwolfach problems were solved in applications 1-3.
\begin{itemize}
\item $OP(2p+1,\ ^{kp}4)$ for each prime $p$ and each $k \in \mathbb{N}$
\item $OP(2p+1,\ ^p6,\ ^{p(2k-2)}4)$ for each prime $p$ and each $k \in \mathbb{N}$
\item $OP(2pr+1,\ ^p2s)$ for all primes $p$ except when $(2r+1,\ 2s) = (3,\ 8),\ (5,\ 4),$ $(7,\ 4),\ (5,\ 6)$ or $(3,\ 2s)$ where $s \equiv 0 \textrm{ or } 1\ \modd(4)$
\item $OP(p2^{\alpha}+1,\ ^{2p}k,\ ^{p2^{\alpha}-1}(2k))$ for each odd prime $p$ when $k$ is odd and $k \geq 3$ and $\alpha \geq 1$
\item $OP(p2^{\alpha}+1,\ ^{p(2^{\alpha}+1)}k)$ for all primes $p$ when $k$ is even and $k \geq 3$ $\alpha \geq 1$
\end{itemize} Of course, many other Oberwolfach problems can be found by applying Theorem \ref{cor2} to other known $1$-rotational Oberwolfach solutions; the above simply lists a few.

\section{Acknowledgements}

We would like to thank the anonymous referees from the Journal of Combinatorial Designs for their constructive feedback on this paper. Their contributions led to a simpler proof for Theorem 3.1, a more concise Section 4, and various other minor edits throughout the paper.


\end{document}